\documentclass[amssymb, 11pt]{amsart}
\usepackage{amsmath}
\usepackage{amsthm}
\usepackage{amsfonts}
\newcommand{\ee}{\mathbb{E}}
\newcommand{\pp}{\mathbb{P}}
\newtheorem{thm}{Theorem}[section]
\newtheorem{theorem}[thm]{Theorem}

\newtheorem{lemma}[thm]{Lemma}

\usepackage{varioref}
\labelformat{section}{Section~#1}
\labelformat{figure}{Figure~#1}
\labelformat{table}{Table~#1}

\usepackage[numbers]{natbib}
\date{}

\begin{document}

\title [Stein's method, heat kernel, and traces of powers] {Stein's method, heat kernel, and traces of powers of elements of
compact Lie groups}

\author{Jason Fulman}
\address{Department of Mathematics\\
        University of Southern California\\
        Los Angeles, CA, 90089, USA}
\email{fulman@usc.edu}


\date{May 4, 2010}

\thanks{Dedicated to Thuy Le, on the occasion of our tenth anniversary.}

\begin{abstract}
Combining Stein's method with heat kernel techniques, we show that the
trace of the $j$th power of an element of $U(n,\mathbb{C}), USp(n,\mathbb{C})$ or $SO(n,\mathbb{R})$
has a normal limit with error term of order $j/n$. In contrast to previous works, here $j$
may be growing with $n$. The technique should prove useful in the study of the value
distribution of approximate eigenfunctions of Laplacians.
\end{abstract}

\maketitle

\section{Introduction} \label{intro}

There is a large literature on the traces of powers of random elements of compact
Lie groups. One of the earliest results is due to Diaconis and Shahshahani
\cite{DS}. Using the method of moments, they show that if $M$ is random
from the Haar measure of the unitary group $U(n,\mathbb{C})$, and $Z=X+iY$ is a
standard complex normal with $X$ and $Y$ independent, mean $0$ and
variance $\frac{1}{2}$ normal variables, then for $j=1,2,\cdots$, the traces
$Tr(M^j)$ are independent and distributed as $\sqrt{j}Z$ asymptotically as
$n \rightarrow \infty$. They give similar results for the orthogonal group
$O(n,\mathbb{R})$ and the group of unitary symplectic matrices $USp(2n,\mathbb{C})$. The
moment computations of \cite{DS} use representation theory. It is worth noting
that there are other approaches to their moment computations: \cite{PV} uses
a version of integration by parts, \cite{HR} uses the combinatorics of
cumulant expansions, and \cite{CS} uses an extended Wick calculus.
We mention that traces of powers of random matrices have been studied for
other matrix ensembles too (\cite{C},\cite{DuE},\cite{So}).

Concerning the error in the normal approximation, Diaconis conjectured
that for fixed $j$, it decreases exponentially or even subexponentially in
$n$. In an ingenious paper (which is quite technical and seems tricky to
apply to other settings), Stein \cite{St2} uses an iterative
version of ``Stein's method'' to show that for $j$ fixed, $Tr(M^j)$ on $O(n,\mathbb{R})$
is asymptotically normal with error $O(n^{-r})$ for any fixed
$r$. Johansson \cite{J} proved Diaconis' conjecture for classical compact
Lie groups using Toeplitz determinants and a very detailed analysis of
characteristic functions. Duits and Johansson \cite{DJ} allow $j$ to grow with $n$
in the unitary case, but do not obtain error terms. We also note that in the
unitary case when $j \geq n$, the situation is not so interesting, since by work of
Rains \cite{R2}, the eigenvalues of $M^j$ are simply $n$ independent points from
the unit circle (and he proves analogous results for other compact Lie groups).

The current paper studies the distribution of $Tr(M^j)$ using Stein's
method and heat kernel techniques. This is a follow-up work to the paper
\cite{F}, which used Stein's method and character theory to study the
distribution of $\chi(M)$, where $\chi$ is the character of an irreducible
representation; the functions $Tr(M^j)$ are not irreducible characters
for $j>1$, so do not fit into the framework of \cite{F}. It should also
be mentioned that the heat kernel is a truly remarkable tool appearing in many
parts of mathematics (see the article \cite{La} for a spirited defense of
this statement with many references), and we suspect that the blending of heat
kernel techniques with Stein's method will be useful for other problems.

In earlier work, Meckes \cite{Me}, used Stein's method to study eigenfunctions of the Laplacian
(a topic of interest in quantum chaos and arithmetic \cite{Sa}, among other places). We note
two differences with her work. First, she uses geodesic flows and Liouville measure instead
of heat kernels. Second, her infinitesimal version of Stein's method \cite{Me}, \cite{Me2} uses
an exchangeable pair of random
variables $(W,W_{\epsilon})$ with the conditional expectation $\ee[W_{\epsilon}-W|W]$
divided by $\epsilon^2$ approximately proportional to $W$ as $\epsilon \rightarrow 0$. In the
current paper the natural condition is that $\ee[W_{\epsilon}-W|W]$
divided by $\epsilon$ is approximately proportional to $W$ as $\epsilon \rightarrow 0$.

We do use some moment computations from \cite{DS}, but as is typical with Stein's method,
only a few low order moments are needed. It should also be mentioned that the constants in
our error terms can be made completely explicit (for instance in the unitary
case we prove a bound of $\frac{22 j}{n}$), but we do not work out the other constants as the
bookkeeping is tedious and the true convergence rate is likely to be of a sharper
order. As to future work, we note that more general linear combinations of traces of powers
do satisfy central limit theorems (see \cite{DE}, \cite{J}, \cite{So2} for precise conditions); obtaining
good error terms by our techniques (or other methods) may be quite tricky and is an important problem.

The organization of this paper is as follows. \ref{background} gives
background on both Stein's method and the heat kernel. \ref{O} treats
the orthogonal groups, \ref{Sp} treats the symplectic groups, and
\ref{U} treats the unitary groups.

\section{Stein's method and the heat kernel} \label{background}

        In this section we briefly review Stein's method for normal
        approximation, using the method of exchangeable pairs
        \cite{Stn}. One can also use couplings to prove normal
        approximations by Stein's method (see \cite{Re} for a survey),
        but the exchangeable pairs approach is effective for our
        purposes. For a survey discussing both exchangeable pairs and
        couplings, the paper \cite{RR} can be consulted.

        Two random variables $W,W'$ on a state space $X$ are called
        exchangeable if the distribution of $(W,W')$ is the same as
        the distribution of $(W',W)$. As is typical in probability
        theory, let $\ee(A|B)$ denote the expected value of $A$ given
        $B$. The following result of Rinott and Rotar \cite{RR2} uses an exchangeable
        pair $(W,W')$ to prove a central limit theorem for $W$.

\begin{theorem} \label{steinbound} (\cite{RR2}) Let $(W,W')$ be an
 exchangeable pair of real random variables such that $\ee(W)=0, \ee(W^2)=1$ and
$\ee(W'|W) = (1-a)W + R(W) $ with $0< a <1$. Then for all real $x_0$,
\begin{eqnarray*} & & \left| \pp(W \leq x_0) - \frac{1}{\sqrt{2 \pi}}
\int_{-\infty}^{x_0} e^{-\frac{x^2}{2}} dx \right|\\ & \leq &
\frac{6}{a} \sqrt{Var(\ee[(W'-W)^2|W])} + 19 \frac{\sqrt{\ee(R^2)}}{a} +
6 \sqrt{\frac{1}{a} \ee|W'-W|^3}. \end{eqnarray*}
\end{theorem}

In practice, it can be quite challenging to construct exchangeable pairs
satisfying the hypotheses of Theorem \ref{steinbound}, and such that the
error terms are tractable and small.

 Lemma \ref{majorize} is a known inequality (already used
    in the monograph \cite{Stn}) and useful because often the
    right hand sides are easier to compute or bound than the left
    hand sides. We include the short proof. Here $M$ is an element of the state space
    $X$ (in this paper $X$ is a compact Lie group and $M$ a matrix in $X$).

\begin{lemma} \label{majorize}
\begin{enumerate}
\item $Var(\ee[(W'-W)^2|W]) \leq Var(\ee[(W'-W)^2|M])$.
\item With notation as in Theorem \ref{steinbound}, letting
$\ee(W'|M) = (1-a)W + R(M)$, one has that $\ee(R(W)^2) \leq \ee(R(M)^2)$.
\end{enumerate}
\end{lemma}

\begin{proof} Jensen's inequality states that if $g$ is a convex function, and $Z$ a random
 variable, then $g(\ee(Z)) \leq \ee(g(Z))$. There is also a
 conditional version of Jensen's inequality (Section 4.1 of \cite{Dur})
 which states that for any $\sigma$ subalgebra ${\it F}$ of the
 $\sigma$-algebra of all subsets of $X$, \[ \ee(g(\ee(Z|{\it F})))
 \leq \ee(g(Z)).\]

 Part 1 now follows by setting $g(t)=t^2$,
 $Z=\ee((W'-W)^2|M)$, and letting ${\it F}$ be the $\sigma$-algebra
 generated by the level sets of $W$. Part 2 follows by setting $g(t)=t^2$,
 $Z=R(M)$, and letting ${\it F}$ be the $\sigma$-algebra
 generated by the level sets of $W$. \end{proof}

To construct an exchangeable pair to be used in our applications, we use the heat kernel on $G$.
See \cite{G}, \cite{Ro}, \cite{S}, \cite{M}) for a detailed discussion of heat kernels on compact
Lie groups, including all of the properties stated in the remainder of this section. The papers
\cite{L}, \cite{B}, \cite{R}, \cite{Liu} illustrate combinatorial uses of heat kernels on compact Lie groups.

The heat kernel on $G$ is defined by setting for $x,y \in G$ and $t \geq 0$, \begin{equation} \label{heat} K(t,x,y) =
\sum_{n \geq 0} e^{-\lambda_n t} \phi_n(x) \overline{\phi_n(y)}, \end{equation} where the $\lambda_n$ are the eigenvalues
of the Laplacian repeated according to multiplicity, and the $\phi_n$ are an orthonormal basis of eigenfunctions of $L^2(G)$; these can be
taken to be the irreducible characters of $G$.

We use the following properties of the heat kernel. Here $\Delta$ denotes the Laplacian of $G$, and $e^{t \Delta}$ is defined as
$I+ t \Delta+ t^2 \frac{\Delta^2}{2!} + \cdots$.

\begin{lemma} \label{spectral} Let $G$ be a compact Lie group, $x,y \in G$, and $t \geq 0$.
\begin{enumerate}
\item $K(t,x,y)$ converges and is non-negative for all $x,y,t$.
\item $\int_{y \in G} K(t,x,y) dy = 1$, where the integration is with respect to Haar measure of $G$.
\item $e^{t \Delta} \phi(x) = \int_{y \in G} K(t,x,y) \phi(y) dy$ for smooth $\phi$.
\end{enumerate}
\end{lemma}

The symmetry in $x$ and $y$ of $K(t,x,y)$ shows that the heat kernel is a reversible Markov
process with respect to the Haar measure of $G$. Thus, given a function $f$ on $G$, one can construct an exchangeable pair $(W,W')$ by
letting $W=f(M)$ where $M$ is chosen from Haar measure, and $W'=f(M')$, where $M'$ is obtained by moving time $t$ from $M$ via the heat kernel.

\section{The orthogonal group} \label{O}

If $\lambda$ is an integer partition (possibly with negative parts) and $m_j$ denotes the multiplicity of part $j$ in $\lambda$,
we define $p_{\lambda}(M)= \prod_j Tr(M^j)^{m_j}$. For example, $p_{5,3,3}(M)=Tr(M^5)Tr(M^3)^2$. Typically we suppress the $M$
and use the notation $p_{\lambda}$. We let $W=\frac{p_j}{\sqrt{j}}$ if $j$ is odd and let $W=\frac{p_j-1}{\sqrt{j}}$ if $j$ is even.
Note that since the eigenvalues of $M$ are roots of unity and come in conjugate pairs, $p_j=p_{-j}$ is real. The main result of this section
is a central limit theorem for $W$ with error term $O(j/n)$.

The following moment computation of \cite{HR} (analogous to that of \cite{DS} for the full orthogonal group) will be helpful.
In fact as the reader will see, in the applications of Lemma \ref{SOpowsum}, we only use fourth moments and lower.

\begin{lemma} \label{SOpowsum} Let $M$ be Haar distributed on $SO(n,\mathbb{R})$. Let $(a_1,a_2,\cdots,a_k)$ be
a vector of non-negative integers. Let $Z_1,\cdots,Z_k$ be independent standard normal random variables. Let $\eta_j$ be
1 if $j$ is even and $0$ otherwise. Then if $n-1 \geq \sum_{i=1}^k a_i$,
\[ \ee \left[ \prod_{j=1}^k Tr(M^j)^{a_j} \right] = \prod_{j=1}^k g_j(a_j) = \prod_{j=1}^k \ee(\sqrt{j} Z_j + \eta_j)^{a_j}, \] Here \[ if \ j \ is \ odd, \ g_j(a) = \left\{ \begin{array}{ll}
0 & if \ a \ is \ odd \\
j^{a/2} (a-1)(a-3) \cdots 1 & if \ a \ is \ even \end{array} \right. \]
\[ if \ j \ is \ even, \ g_j(a)=1 + \sum_{k \geq 1} {a \choose 2k} j^k (2k-1)(2k-3) \cdots 1.\]
\end{lemma}

Rains \cite{R} (see also \cite{L}) determined how the Laplacian acts on power sum symmetric functions. We need his formula
only in the following two cases.

\begin{lemma} \label{SOform}
\begin{enumerate}
\item \[ \Delta_{SO(n)} p_j = - \frac{(n-1)j}{2} p_j - \frac{j}{2} \sum_{1 \leq l <j} p_{l,j-l} + \frac{j}{2} \sum_{1 \leq l <j} p_{2l-j}.\]
\item \[ \Delta_{SO(n)} p_{j,j} = - (n-1)jp_{j,j} - j^2 p_{2j} - j p_j \sum_{1 \leq l <j} p_{l,j-l} + j p_j \sum_{1 \leq l<j} p_{2l-j} + j^2 n.\]
\end{enumerate}
\end{lemma}

We fix $t>0$, and motivated by \ref{background}, define \[ W' = e^{t \Delta}(W) = W + \sum_{k \geq 1} \frac{t^k}{k!} \Delta^k(W).\]

Lemma \ref{SOcond1} computes the conditional expectation $\ee[W'|M]$.

\begin{lemma} \label{SOcond1} \[ \ee[W'|M] = \left( 1 - \frac{t(n-1)j}{2} \right) W + R(M), \] with \[ R(M) = t \left[ - \frac{\sqrt{j}}{2} \sum_{1 \leq l <j} p_{l,j-l} + \frac{\sqrt{j}}{2} \sum_{1 \leq l < j} p_{2l-j} \right] + O(t^2) \ \ \ j \ odd, \] and \[ R(M) = t \left[ - \frac{(n-1)\sqrt{j}}{2} - \frac{\sqrt{j}}{2} \sum_{1 \leq l <j} p_{l,j-l} + \frac{\sqrt{j}}{2} \sum_{1 \leq l < j} p_{2l-j} \right] + O(t^2) \ \ \ j \ even. \]
\end{lemma}

\begin{proof} Applying part 3 of Lemma \ref{spectral} and part 1 of Lemma \ref{SOform},
\begin{eqnarray*} & & \ee[W'|M]\\
& = & e^{t \Delta}(W)\\
& = &  W + t \left[ -\frac{(n-1)\sqrt{j}}{2} p_j - \frac{\sqrt{j}}{2} \sum_{1 \leq l <j} p_{l,j-l} + \frac{\sqrt{j}}{2} \sum_{1 \leq l < j} p_{2l-j} \right] + O(t^2), \end{eqnarray*} and the result follows. \end{proof}

Lemma \ref{SOcond2} computes $\ee[(W'-W)^2|M]$, a quantity needed to apply Theorem \ref{steinbound}. Many cancelations occur,
and a simple formula emerges.

\begin{lemma} \label{SOcond2}
\[ \ee[(W'-W)^2|M] = tj (n - p_{2j}) + O(t^2).\]
\end{lemma}

\begin{proof} Clearly
\[ \ee[(W'-W)^2|M] = \ee[(W')^2|M] - 2W \ee[W'|M] + W^2.\] Suppose now that $j$ is odd.
By part 3 of Lemma \ref{spectral} and part 2 of Lemma \ref{SOform},
\begin{eqnarray*}
& & \ee[(W')^2|M] \\
 & = & W^2 + \frac{t}{j} \Delta p_{j,j} + O(t^2) \\
& = & W^2 + t \left[ -(n-1)p_{j,j} - j p_{2j} - p_j \sum_{1 \leq l <j} p_{l,j-l} + p_j \sum_{1 \leq l<j} p_{2l-j} + jn \right]\\
& & + O(t^2).
\end{eqnarray*} By Lemma \ref{SOcond1}, $-2 W \ee[W'|M]$ is equal to
\[ -2W^2 + t \left[ (n-1)j W^2 + p_j \sum_{1 \leq l <j} p_{l,j-l} - p_j \sum_{1 \leq l < j} p_{2l-j} \right] + O(t^2).\]
Thus \[ \ee[(W')^2|M] - 2W \ee[W'|M] + W^2 = tj (n - p_{2j}) + O(t^2),\] as claimed. A very similar calculation
shows that the same conclusion holds for $j$ even.
\end{proof}

\begin{lemma} \label{Sovar} Suppose that $4j \leq n-1$. Then
\[ Var(\ee[(W'-W)^2|M]) = 2j^3t^2 + O(t^3).\]
\end{lemma}

\begin{proof} By Lemma \ref{SOcond2}, \[ Var(\ee[(W'-W)^2|M]) = j^2 t^2 Var(p_{2j})+O(t^3).\] The result
now follows from Lemma \ref{SOpowsum}. \end{proof}

\begin{lemma} \label{SOlow} Suppose that $4j \leq n-1$. Then
\begin{enumerate}
\item $\ee(W'-W)^2 = tj (n-1) + O(t^2)$.
\item $\ee(W'-W)^4 = O(t^2)$.
\end{enumerate}
\end{lemma}

\begin{proof} Lemma \ref{SOcond2} implies that $\ee(W'-W)^2 = \ee[tj(n-p_{2j})]+O(t^2)$. From Lemma \ref{SOpowsum},
$\ee(p_{2j})=1$, which proves part 1 (only using that $2j \leq n-1$).

For part 2, first note that since
\[ \ee[(W'-W)^4] = \ee(W^4) - 4 \ee(W^3W') + 6 \ee[W^2(W')^2] - 4 \ee[W (W')^3] + \ee[(W')^4],\]
exchangeability of $(W,W')$ gives that
\begin{eqnarray*}
\ee(W'-W)^4 & = & 2 \ee(W^4) -8 \ee(W^3W') + 6 \ee[W^2(W')^2] \\
& = & 2 \ee(W^4) -8 \ee[W^3 \ee[W'|M]] + 6 \ee[W^2 \ee[(W')^2|M]].
\end{eqnarray*}

Supposing that $j$ is odd and using Lemma \ref{SOform}, this simplifies to
\begin{eqnarray*}
& & 2 \ee(W^4) - 8 \ee[W^4] + 6 \ee[W^4] \\
& & + t \ee \left[ 4 (n-1)j W^4 + 4 W^3 \sqrt{j} \sum_{1 \leq l < j} p_{l,j-l}- 4 W^3 \sqrt{j} \sum_{1 \leq l < j} p_{2l-j} \right] \\
& &  + t \ee \left[ -6(n-1)jW^4 - 6 W^2 p_j \sum_{1 \leq l<j} p_{l,j-l}+ 6 W^2 p_j \sum_{1 \leq l <j} p_{2l-j} \right] \\
& & + t \ee \left[ - 6j W^2 p_{2j} + 6 W^2 j n \right] + O(t^2).
\end{eqnarray*} By Lemma \ref{SOpowsum}, this simplifies to
\[ t \left[ 12j(n-1) - 18j(n-1)-6j+6jn \right] + O(t^2) = O(t^2),\] as claimed. A very similar calculation gives the same conclusion
for $j$ even.
\end{proof}

Next we bound a quantity appearing in the second term of Theorem \ref{steinbound}.

\begin{lemma} \label{SOboundR} Suppose that $4j \leq n-1$. Let \[ R(M)= t \left[ - \frac{\sqrt{j}}{2} \sum_{1 \leq l <j} p_{l,j-l} + \frac{\sqrt{j}}{2} \sum_{1 \leq l < j} p_{2l-j} \right] + O(t^2) \ \ \ j \ odd, \] and \[ R(M)= t \left[ - \frac{(n-1)\sqrt{j}}{2} - \frac{\sqrt{j}}{2} \sum_{1 \leq l <j} p_{l,j-l} + \frac{\sqrt{j}}{2} \sum_{1 \leq l < j} p_{2l-j} \right] + O(t^2) \ \ \ j \ even. \] Then $\ee[R^2]=O(t^2j^4)$. \end{lemma}

\begin{proof} Suppose that $j$ is odd. Applying Lemma \ref{SOpowsum} and keeping only terms with non-0 expectations, one has that
\begin{eqnarray*} \ee[R^2] & = & \frac{t^2j}{4} \ee \left[ 4 \sum_{1 \leq l < j \atop l \ odd} (p_{l,j-l})^2 + 4 \sum_{1 \leq l < j \atop l \ odd} (p_l)^2 - 8 \sum_{1 \leq l<j \atop l \ odd} p_{l,l,j-l} \right] + O(t^3) \\
& = & \frac{t^2j}{4} \left[ 4 \sum_{1 \leq l<j \atop l \ odd} l(j-l+1) - 4 \sum_{1 \leq l<j \atop l \ odd} l\right] + O(t^3) \\
& = & O(t^2j^4). \end{eqnarray*} The case of $j$ even is proved in a similar way, as can be seen by writing
\[ R = t \left[ \frac{\sqrt{j}}{2} - \frac{\sqrt{j}}{2} \sum_{1 \leq l <j} p_{l,j-l} + \frac{\sqrt{j}}{2} \sum_{1 \leq l < j \atop l \neq j/2} p_{2l-j} \right] + O(t^2). \]
\end{proof}

Combining the above calculations leads to the main result of this section.

\begin{theorem} Let $M$ be chosen from the Haar measure of $SO(n,\mathbb{R})$. Let $W(M)= \frac{Tr(M^j)}{\sqrt{j}}$ if $j$ is odd and $W(M)= \frac{Tr(M^j)-1}{\sqrt{j}}$ if $j$ is even. Then  \[ \left| \pp(W \leq x_0) - \frac{1}{\sqrt{2 \pi}} \int_{-\infty}^{x_0} e^{-\frac{x^2}{2}} dx \right| = O(j/n).\]
\end{theorem}

\begin{proof} The result is trivial if $4j > n-1$, so assume that $4j \leq n-1$. We apply Theorem \ref{steinbound} to the exchangeable pair $(W,W')$ with $a=\frac{t(n-1)j}{2}$, and will take the limit $t \rightarrow 0$ in each term (keeping $j,n$ fixed). By part 1 of Lemma \ref{majorize} and Lemma \ref{Sovar}, the first term is $O(\sqrt{j}/n)$. By part 2 of Lemma \ref{majorize} and Lemma \ref{SOboundR}, the second term is $O(j/n)$. By the Cauchy-Schwarz inequality and Lemma \ref{SOlow}, \[ \ee|W'-W|^3 \leq \sqrt{\ee(W'-W)^2 \ee(W'-W)^4} = O(t^{3/2}) .\] Thus the third term in Theorem \ref{steinbound} tends to $0$ as $t \rightarrow 0$, and the result is proved. \end{proof}

\section{The symplectic group} \label{Sp}

Let $J$ be the $2n \times 2n$ matrix of the form $\left( \begin{array}{c c} 0 & I \\ -I & 0 \end{array} \right)$ with all blocks $n \times n$.
$USp(2n,\mathbb{C})$ is defined as the set of $2n \times 2n$ unitary matrices $M$ with complex entries such that $MJM^t=J$; it consists of the matrices preserving
an alternating form. As in \ref{O}, we use the notation that $p_{\lambda}(M)= \prod_j Tr(M^j)^{m_j}$, and we typically suppress the $M$
and use the notation $p_{\lambda}$. We let $W=\frac{p_j}{\sqrt{j}}$ if $j$ is odd and let $W=\frac{p_j+1}{\sqrt{j}}$ if $j$ is even. Since the eigenvalues of $M$
are roots of unity and come in conjugate pairs, $p_j=p_{-j}$ is real valued. The main result of this section is a central limit theorem for $W$ with error term $O(j/n)$.

The following moment computation is the symplectic analog of Lemma \ref{SOpowsum}. It was proved by \cite{DS} under the slightly weaker
assumption that $n \geq \sum_{i=1}^k a_k$. As stated, Lemma \ref{Sppowsum} appears in \cite{HR}, with a later proof in \cite{PV}.

\begin{lemma} \label{Sppowsum} Let $M$ be Haar distributed on $USp(2n,\mathbb{C})$. Let $(a_1,a_2,\cdots,a_k)$ be
a vector of non-negative integers. Let $Z_1,\cdots,Z_k$ be independent standard normal random variables. Let $\eta_j$ be
1 if $j$ is even and $0$ otherwise. Then if $2n+1 \geq \sum_{i=1}^k a_i$,
\[ \ee \left[ \prod_{j=1}^k Tr(M^j)^{a_j} \right] = \prod_{j=1}^k (-1)^{(j-1)a_j} g_j(a_j) = \prod_{j=1}^k \ee(\sqrt{j} Z_j - \eta_j)^{a_j}, \]
where the polynomials $g_j$ are as in Lemma \ref{SOpowsum}. \end{lemma}

Rains \cite{R} (see also \cite{L}) determined how the Laplacian acts on power sum symmetric functions. We need his formula
only in the following two cases.

\begin{lemma} \label{Spform}
\begin{enumerate}
\item \[ \Delta_{USp(2n)} p_j = - \frac{(2n+1)j}{2} p_j - \frac{j}{2} \sum_{1 \leq l <j} p_{2l-j} - \frac{j}{2} \sum_{1 \leq l < j} p_{l,j-l}.\]
\item \[ \Delta_{USp(2n)} p_{j,j} = - (2n+1)j p_{j,j} - j^2 p_{2j} - j p_j \sum_{1 \leq l<j} p_{2l-j} - j p_j \sum_{1 \leq l <j} p_{l,j-l}  + 2j^2n.\]
\end{enumerate}
\end{lemma}

As in the orthogonal case, we fix $t>0$, and define \[ W' = e^{t \Delta}(W) = W + \sum_{k \geq 1} \frac{t^k}{k!} \Delta^k(W).\]

\begin{lemma} \label{Spcond1} \[ \ee[W'|M] = \left( 1 - \frac{t(2n+1)j}{2} \right) W + R(M), \] with \[ R(M) = t \left[ - \frac{\sqrt{j}}{2} \sum_{1 \leq l < j} p_{2l-j} - \frac{\sqrt{j}}{2} \sum_{1 \leq l <j} p_{l,j-l}  \right] + O(t^2) \ \ \ j \ odd, \] and \[ R(M) = t \left[  \frac{(2n+1) \sqrt{j}}{2} - \frac{\sqrt{j}}{2} \sum_{1 \leq l < j} p_{2l-j} - \frac{\sqrt{j}}{2} \sum_{1 \leq l <j} p_{l,j-l} \right] + O(t^2) \ \ \ j \ even. \]
\end{lemma}

\begin{proof} Applying part 3 of Lemma \ref{spectral} and part 1 of Lemma \ref{Spform},
\begin{eqnarray*} & & \ee[W'|W]\\
& = & e^{t \Delta}(W)\\
& = &  W + t \left[ -\frac{(2n+1)\sqrt{j}}{2} p_j - \frac{\sqrt{j}}{2} \sum_{1 \leq l < j} p_{2l-j} - \frac{\sqrt{j}}{2} \sum_{1 \leq l <j} p_{l,j-l}  \right] + O(t^2), \end{eqnarray*} and the result follows. \end{proof}

Lemma \ref{Spcond2} computes $\ee[(W'-W)^2|M]$, a quantity needed to apply Theorem \ref{steinbound}. As in the orthogonal case, there are many cancelations,
leading to a simple formula.

\begin{lemma} \label{Spcond2}
\[ \ee[(W'-W)^2|M] = tj \left(2n - p_{2j} \right) + O(t^2).\]
\end{lemma}

\begin{proof} Clearly
\[ \ee[(W'-W)^2|M] = \ee[(W')^2|M] - 2W \ee[W'|M] + W^2.\]
Suppose that $j$ is odd. By part 3 of Lemma \ref{spectral} and part 2 of Lemma \ref{Spform},
\begin{eqnarray*}
& & \ee[(W')^2|M] \\
 & = & W^2 + \frac{t}{j} \Delta p_{j,j} + O(t^2) \\
& = & W^2 + t \left[ -(2n+1)p_{j,j} - j p_{2j} - p_j \sum_{1 \leq l<j} p_{2l-j} - p_j \sum_{1 \leq l <j} p_{l,j-l} + 2jn \right]\\
& & + O(t^2).
\end{eqnarray*} By Lemma \ref{Spcond1}, $-2 W \ee[W'|M]$ is equal to
\[ -2W^2 + t \left[ (2n+1)p_{j,j} + p_j \sum_{1 \leq l < j} p_{2l-j} + p_j \sum_{1 \leq l <j} p_{l,j-l}  \right] + O(t^2).\]
Thus \[ \ee[(W')^2|M] - 2W \ee[W'|M] + W^2 = tj \left[ 2n - p_{2j} \right] + O(t^2),\] as needed. A similar computation
proves the lemma for $j$ even.
\end{proof}

\begin{lemma} \label{Spvar} Suppose that $4j \leq 2n+1$. Then
\[ Var(\ee[(W'-W)^2|M]) = 2j^3t^2 + O(t^3).\]
\end{lemma}

\begin{proof} By Lemma \ref{Spcond2}, \[ Var(\ee[(W'-W)^2|M]) = j^2 t^2 Var(p_{2j})+O(t^3).\] The result
now follows from Lemma \ref{Sppowsum}. \end{proof}

\begin{lemma} \label{Splow} Suppose that $4j \leq 2n+1$.
\begin{enumerate}
\item $\ee(W'-W)^2 = tj (2n+1) + O(t^2)$.
\item $\ee(W'-W)^4 = O(t^2)$.
\end{enumerate}
\end{lemma}

\begin{proof} Lemma \ref{Spcond2} implies that $\ee(W'-W)^2 = \ee \left[ tj \left( 2n - p_{2j} \right) \right]+O(t^2)$. From Lemma \ref{Sppowsum},
$\ee(p_{2j})=-1$, which proves part 1 (even assuming that $2j \leq 2n+1$).

For part 2, first note that since
\[ \ee[(W'-W)^4] = \ee(W^4) - 4 \ee(W^3W') + 6 \ee[W^2(W')^2] - 4 \ee[W (W')^3] + \ee[(W')^4],\]
exchangeability of $(W,W')$ gives that
\begin{eqnarray*}
\ee(W'-W)^4 & = & 2 \ee(W^4) -8 \ee(W^3W') + 6 \ee[W^2(W')^2] \\
& = & 2 \ee(W^4) -8 \ee[W^3 \ee[W'|M]] + 6 \ee[W^2 \ee[(W')^2|M]].
\end{eqnarray*}

Suppose $j$ is odd. Using Lemma \ref{Spcond1} and part 2 of Lemma \ref{Spform}, this becomes
\begin{eqnarray*}
& & 2 \ee(W^4) - 8 \ee(W^4) + 6 \ee(W^4) \\
& & + t \ee \left[ 4 (2n+1)j W^4 + 4  W^3 \sqrt{j} \sum_{1 \leq l < j} p_{l,j-l} + 4 W^3 \sqrt{j} \sum_{1 \leq l < j} p_{2l-j} \right] \\
& &  + t \ee \left[ - 6(2n+1)j W^4 - 6j W^2 p_{2j} - 6 W^2 p_j \sum_{1 \leq l<j} p_{l,j-l} \right] \\ & & + t \left[ -6 W^2p_j \sum_{1 \leq l <j} p_{2l-j} +  12 W^2 j n \right] + O(t^2).
\end{eqnarray*} By Lemma \ref{Sppowsum}, this simplifies to
\[ t \left[ 12j(2n+1) - 18j(2n+1)+6j+12jn \right] + O(t^2) = O(t^2),\] as claimed. A similar calculation gives the same result for $j$ even.
\end{proof}

\begin{lemma} \label{SpboundR} Suppose that $4j \leq 2n+1$. Let \[ R(M)= t \left[ - \frac{\sqrt{j}}{2} \sum_{1 \leq l < j} p_{2l-j} - \frac{\sqrt{j}}{2} \sum_{1 \leq l <j} p_{l,j-l}  \right] + O(t^2) \ \ \ j \ odd, \] and \[ R(M)= t \left[  \frac{(2n+1) \sqrt{j}}{2} - \frac{\sqrt{j}}{2} \sum_{1 \leq l < j} p_{2l-j} -\frac{\sqrt{j}}{2} \sum_{1 \leq l <j} p_{l,j-l} \right] + O(t^2) \ \ \ j \ even. \] Then $\ee[R^2]=O(t^2j^4)$. \end{lemma}

\begin{proof} Suppose that $j$ is odd. Applying Lemma \ref{Sppowsum} and keeping only terms with non-0 contribution, one has that
\begin{eqnarray*} \ee[R^2] & = & \frac{t^2j}{4} \ee \left[ 4 \sum_{1 \leq l < j \atop l \ odd} (p_{l,j-l})^2 +  4 \sum_{1 \leq l < j \atop l \ odd} (p_l)^2 + 8 \sum_{1 \leq l <j \atop l \ odd} p_{l,l,j-l} \right] + O(t^3) \\
& = & \frac{t^2j}{4} \left[ 4 \sum_{1 \leq l < j \atop l \ odd} l(j-l+1) - 4 \sum_{1 \leq l < j \atop l \ odd} l \right] + O(t^3)\\
& = & O(t^2j^4). \end{eqnarray*} The case of $j$ even is proved by a similar argument, after writing \[ R=  t \left[  \frac{\sqrt{j}}{2} - \frac{\sqrt{j}}{2} \sum_{1 \leq l < j \atop l \neq j/2} p_{2l-j} - \frac{\sqrt{j}}{2} \sum_{1 \leq l <j} p_{l,j-l} \right] + O(t^2) . \]
\end{proof}

\begin{theorem} Let $M$ be chosen from the Haar measure of $USp(2n,\mathbb{C})$. Let $W(M)=\frac{Tr(M^j)}{\sqrt{j}}$ if $j$ is odd, and $W(M)=\frac{Tr(M^j)+1}{\sqrt{j}}$
if $j$ is even. Then  \[ \left| \pp(W \leq x_0) - \frac{1}{\sqrt{2 \pi}} \int_{-\infty}^{x_0} e^{-\frac{x^2}{2}} dx \right| = O(j/n).\]
\end{theorem}

\begin{proof} The result is trivial if $4j > 2n+1$, so assume that $4j \leq 2n+1$. We apply Theorem \ref{steinbound} to the exchangeable pair $(W,W')$ with $a=\frac{t(2n+1)j}{2}$, and will take the limit $t \rightarrow 0$ in each term (keeping $j,n$ fixed). By part 1 of Lemma \ref{majorize} and Lemma \ref{Spvar}, the first term is $O(\sqrt{j}/n)$. By part 2 of Lemma \ref{majorize} and Lemma \ref{SpboundR}, the second term is $O(j/n)$. By the Cauchy-Schwarz inequality and Lemma \ref{Splow}, \[ \ee|W'-W|^3 \leq \sqrt{\ee(W'-W)^2 \ee(W'-W)^4} = O(t^{3/2}) .\] Thus the third term in Theorem \ref{steinbound} tends to $0$ as $t \rightarrow 0$, and the result follows.
\end{proof}

\section{The unitary group} \label{U}

In this final section, we treat the unitary group $U(n,\mathbb{C})$. We let $p_{\lambda}$ be as in \ref{O} and \ref{Sp} and define the real valued random
variable $W=\frac{p_j+\overline{p_j}}{\sqrt{2j}}$. The main result of this section is a central limit theorem for $W$, with error term $O(j/n)$. To begin, we recall the following moment computation from \cite{DS}.

\begin{lemma} \label{Upowsum} Let $M$ be Haar distributed on $U(n,\mathbb{C})$. Let $(a_1,a_2,\cdots,a_k)$ and $(b_1,\cdots,b_k)$ be
vectors of non-negative integers. Let $Z_1,\cdots,Z_k$ be independent standard normal random variables. Then for all $n \geq \sum_{i=1}^k (a_i+b_i)$,
\[ \ee \left[ \prod_{j=1}^k Tr(M^j)^{a_j} \cdot \overline{Tr(M^j)}^{b_j} \right] = \delta_{\vec{a} \vec{b}} \prod_{j=1}^k j^{a_j} a_j! .\]
\end{lemma}

Rains \cite{R} (see also \cite{L}) determined how the Laplacian acts on power sum symmetric functions. We require his formulas only in the following cases.

\begin{lemma} \label{Uform}

\begin{enumerate}
\item \[ \Delta_{U(n)} p_j = - nj p_j - j \sum_{1 \leq l <j} p_{l,j-l}.\]
\item \[ \Delta_{U(n)} p_{j,j} = - 2nj p_{j,j} - 2j^2 p_{2j} - 2j p_j \sum_{1 \leq l <j} p_{l,j-l} .\]
\item \[ \Delta_{U(n)} \left( p_j \overline{p_j} \right) = 2j^2n - 2njp_j \overline{p_j} - jp_j \sum_{1 \leq l < j} \overline{p_{l,j-l}} - j \overline{p_j} \sum_{1 \leq l < j} p_{l,j-l} .\]
\end{enumerate}
\end{lemma}

Lemma \ref{Ucond1} computes the conditional expectation $\ee[W'|M]$.

\begin{lemma} \label{Ucond1} \[ \ee[W'|M] = \left( 1 - njt \right) W + R(M), \] with \[ R(M)= t \left[ - \sqrt{\frac{j}{2}} \sum_{1 \leq l <j} p_{l,j-l} - \sqrt{\frac{j}{2}} \sum_{1 \leq l < j} \overline{p_{l,j-l}} \right] + O(t^2) .\]
\end{lemma}

\begin{proof} Applying Lemma \ref{spectral} and part 1 of Lemma \ref{Uform} gives that
\begin{eqnarray*} & & \ee[W'|M]\\
& = & e^{t \Delta}(W)\\
& = &  W + t \left[ - nj W - \sqrt{\frac{j}{2}} \sum_{1 \leq l <j} p_{l,j-l} - \sqrt{\frac{j}{2}} \sum_{1 \leq l < j} \overline{p_{l,j-l}} \right] + O(t^2), \end{eqnarray*} as desired. \end{proof}

Lemma \ref{Ucond2} computes $\ee[(W'-W)^2|M]$. As in the other cases, there are nice cancelations.

\begin{lemma} \label{Ucond2}
\[ \ee[(W'-W)^2|M] = tj \left( 2n - p_{2j} - \overline{p_{2j}} \right) + O(t^2).\]
\end{lemma}

\begin{proof} Clearly
\[ \ee[(W'-W)^2|M] = \ee[(W')^2|M] - 2W \ee[W'|M] + W^2.\]
By Lemmas \ref{spectral} and \ref{Uform},
\begin{eqnarray*}
& & \ee[(W')^2|M] \\
 & = & W^2 + \frac{t}{2j} \Delta [p_{j,j} +2 p_j \overline{p_j} + \overline{p_{j,j}}] + O(t^2) \\
& = & W^2 + t \left[ -np_{j,j} - j p_{2j} - p_j \sum_{1 \leq l<j} p_{l,j-l}
-n \overline{p_{j,j}} - j \overline{p_{2j}}  \right]\\
& & + t \left[ - \overline{p_j} \sum_{1 \leq l<j} \overline{p_{l,j-l}} + 2jn - 2np_j \overline{p_j} - p_j \sum_{1 \leq l < j} \overline{p_{l,j-l}}
- \overline{p_j} \sum_{1 \leq l < j} p_{l,j-l} \right] \\ & & + O(t^2).
\end{eqnarray*} By Lemma \ref{Ucond1}, $-2 W \ee[W'|M]$ is equal to
\begin{eqnarray*}
 &  & -2W^2 + t \left[ n p_{j,j} + 2 n p_j \overline{p_j} + n \overline{p_{j,j}} + p_j \sum_{1 \leq l < j} p_{l,j-l} \right] \\
 & & + t \left[ p_j \sum_{1 \leq l < j} \overline{p_{l,j-l}} + \overline{p_j} \sum_{1 \leq l < j} p_{l,j-l} +  \overline{p_j} \sum_{1 \leq l < j} \overline{p_{l,j-l}} \right] + O(t^2).
\end{eqnarray*} Thus \[ \ee[(W')^2|M] - 2W \ee[W'|M] + W^2 = tj \left[ 2n - p_{2j} - \overline{p_{2j}} \right] + O(t^2),\] and the lemma is proved.
\end{proof}

\begin{lemma} \label{Uvar} Suppose that $4j \leq n$. Then
\[ Var(\ee[(W'-W)^2|M]) = 4j^3t^2 + O(t^3).\]
\end{lemma}

\begin{proof} By Lemmas \ref{Ucond2} and \ref{Upowsum},
\begin{eqnarray*}
Var(\ee[(W'-W)^2|M]) & = & j^2 t^2 Var(p_{2j}+\overline{p_{2j}}) + O(t^3)\\
& = & j^2 t^2 \ee[(p_{2j}+\overline{p_{2j}})^2] + O(t^3)\\
& = & 4j^3t^2 + O(t^3). \end{eqnarray*} \end{proof}

\begin{lemma} \label{Ulow} Suppose that $4j \leq n$.
\begin{enumerate}
\item $\ee(W'-W)^2 = t2jn + O(t^2)$.
\item $\ee(W'-W)^4 = O(t^2)$.
\end{enumerate}
\end{lemma}

\begin{proof} Lemma \ref{Ucond2} implies that $\ee(W'-W)^2 = \ee \left[ tj \left( 2n - p_{2j} - \overline{p_{2j}} \right)  \right]+O(t^2)$. From Lemma \ref{Upowsum}, $\ee(p_{2j})=\ee(\overline{p_{2j}})=0$, which proves part 1 (using only that $2j \leq n$).

For part 2, first note that since
\[ \ee[(W'-W)^4] = \ee(W^4) - 4 \ee(W^3W') + 6 \ee[W^2(W')^2] - 4 \ee[W (W')^3] + \ee[(W')^4],\]
exchangeability of $(W,W')$ gives that
\begin{eqnarray*}
\ee(W'-W)^4 & = & 2 \ee(W^4) -8 \ee(W^3W') + 6 \ee[W^2(W')^2] \\
& = & 2 \ee(W^4) -8 \ee[W^3 \ee[W'|M]] + 6 \ee[W^2 \ee[(W')^2|M]].
\end{eqnarray*} Using Lemmas \ref{Uform} and \ref{Ucond1}, this simplifies to
\begin{eqnarray*}
& & 2 \ee(W^4) - 8 \ee[W^4] + 6 \ee[W^4] \\
& & + t \ee \left[ 8 nj W^4 + 8 W^3 \sqrt{\frac{j}{2}} \sum_{1 \leq l < j} p_{l,j-l} + 8 W^3 \sqrt{\frac{j}{2}} \sum_{1 \leq l < j} \overline{p_{l,j-l}} \right] \\
& & + t \ee \left[ -6n W^2 p_{j,j} - 6jW^2 p_{2j} - 6W^2 p_j \sum_{1 \leq l < j} p_{l,j-l} -6n W^2 \overline{p_{j,j}} \right] \\
& & + t \ee \left[  - 6jW^2 \overline{p_{2j}} - 6W^2 \overline{p_j} \sum_{1 \leq l < j} \overline{p_{l,j-l}} + 12 nj W^2\right] \\
& & + t \ee \left[ -12 n W^2 p_j \overline{p_j} - 6 W^2 p_j \sum_{1 \leq l <j} \overline{p_{l,j-l}} - 6 W^2 \overline{p_j} \sum_{1 \leq l<j} p_{l,j-l} \right] + O(t^2).
\end{eqnarray*} By Lemma \ref{Upowsum}, after dropping out terms with 0 expectation, there remains
\begin{eqnarray*} & & t \ee[ 8 W^4 jn - 6 W^2 n p_{j,j} - 6 W^2 n \overline{p_{j,j}} + 12 W^2 jn - 12 W^2 n p_j \overline{p_j} ] + O(t^2) \\
& = & t [ 24jn - 6jn - 6jn + 12jn -24jn ] + O(t^2) \\
& = & O(t^2), \end{eqnarray*} as needed.
\end{proof}

\begin{lemma} \label{UboundR} Let $R=t \left[ - \sqrt{\frac{j}{2}} \sum_{1 \leq l <j} p_{l,j-l} - \sqrt{\frac{j}{2}} \sum_{1 \leq l <j} \overline{p_{l,j-l}} \right] + O(t^2)$, and suppose that $4j \leq n$. Then $\ee[R^2] \leq \frac{j^4 t^2}{4} + O(t^3)$. \end{lemma}

\begin{proof} Applying Lemma \ref{Upowsum} and keeping only terms with non-0 contribution, one has that
\[\ee[R^2] = jt^2 \ee[ \sum_{1 \leq l < j} p_{l,j-l} \overline{p_{l,j-l}}] + O(t^3). \] If $j$ is odd, then by Lemma \ref{Upowsum}, \[ \ee[R^2] = jt^2 \sum_{1 \leq l < j} [ l (j-l)] + O(t^3) = \frac{(j^4-j^2)}{6} t^2 + O(t^3), \] while if $j$ is even, one obtains that \[ \ee[R^2] = \frac{(2j^4+3j^3-2j^2)}{12} t^2 + O(t^3).\] The result follows.
\end{proof}

\begin{theorem} Let $M$ be chosen from the Haar measure of $U(n,\mathbb{C})$, and let $W(M)=\frac{1}{\sqrt{2j}} [Tr(M^j) + \overline{Tr(M^j)}]$. Then  \[ \left| \pp(W \leq x_0) - \frac{1}{\sqrt{2 \pi}} \int_{-\infty}^{x_0} e^{-\frac{x^2}{2}} dx \right| = O(j/n).\]
\end{theorem}

\begin{proof} The result is trivial if $4j >n$, so assume that $4j \leq n$. We apply Theorem \ref{steinbound} to the exchangeable pair $(W,W')$ with $a=tnj$, and will take the limit $t \rightarrow 0$ in each term. By part 1 of Lemma \ref{majorize} and Lemma \ref{Uvar}, the first term is at most $\frac{12 \sqrt{j}}{n}$. By Lemma \ref{UboundR} and part 2 of Lemma \ref{majorize}, the second term in Theorem \ref{steinbound} is at most $\frac{19 j}{2n}$. By the Cauchy-Schwarz inequality and Lemma \ref{Ulow}, \[ \ee|W'-W|^3 \leq \sqrt{\ee(W'-W)^2 \ee(W'-W)^4} = O(t^{3/2}) .\] Thus the third term in Theorem \ref{steinbound} tends to $0$ as $t \rightarrow 0$, and the result follows since \[ \frac{12 \sqrt{j}}{n} + \frac{19 j}{2n} \leq \frac{22j}{n}.\] \end{proof}

\section*{Acknowledgements} We thank Eric Rains for helpful correspondence. The author was partially supported by NSF grant DMS 0802082 and NSA grant
H98230-08-1-0133.

\end{document}